\documentclass{amsart}
\usepackage{amssymb}
\usepackage{amsfonts}
\usepackage{amsmath}
\usepackage{lastpage}
\usepackage[legalpaper,bookmarks=true,colorlinks=true,linkcolor=blue,citecolor=blue]{hyperref}
\usepackage{graphicx}%
\usepackage{fancyhdr}
\usepackage{color}
\usepackage[mathlines]{lineno}
\usepackage{lscape}
\usepackage{epsfig}
\usepackage{natbib}
\usepackage{geometry}
\usepackage{tgbonum}
\fontfamily{qcr}\selectfont

\newtheorem{theorem}{Theorem}
\theoremstyle{plain}

\newtheorem{lemma}{Lemma}

\newtheorem{proposition}{Proposition}
\newtheorem{remark}{Remark}

\numberwithin{equation}{section}

\begin{document}
\title[Glivenko-Cantelli classes]{On the Glivenko-Cantelli theorem for real-valued empirical functions of stationary $\alpha$-mixing and $\beta$-mixing sequences}


\author{Ousmane Coulibaly$^{1}$}
\author{Harouna Sangar\'e$^{2}$}


\begin{abstract}  
In this paper we extend the classical Glivenko-Cantelli theorem to real-valued empirical functions under dependence structures characterised by $\alpha$-mixing and $\beta$-mixing conditions. We investigate sufficient conditions ensuring that families of real-valued functions exhibit the Glivenko-Cantelli (GC) property in these dependence settings. Our analysis focuses on function classes satisfying uniform entropy conditions and establishes deviation bounds under mixing coefficients that decay at appropriate rates. Our results refine the existing literature by relaxing the independence assumptions and highlighting the role of dependence in empirical process convergence.

\bigskip
\noindent $^{1}$ Msc in Probabilties and Statistics\\
\noindent Universit\'e des Sciences, des Techniques et des Technologies de Bamako (USTT-B), Mali\\
\noindent Personal email : ocouwi@gmail.com\\

\noindent $^{2}$ DER MI, Facult\'e des Sciences et Techniques (FST), USTT-B, Mali \\
\noindent LERSTAD, Gaston Berger University, Saint-Louis, S\'en\'egal.\newline
Researcher at IMHO [\url{https://imhotepsciences.org}]\\
Imhotep International Mathematical Center (imh-imc), \url{https://imhotepsciences.org}\\
\noindent Institutional emails : harouna.sangare@usttb.edu.ml, harouna.sangare@mesrs.ml, \\ sangare.harouna@ugb.edu.sn\\
\noindent Personal email : harounasangareusttb@gmail.com\\


\end{abstract}

\maketitle

\section{Introduction} \label{sec01}

\noindent Empirical process theory plays a crucial role in probability and statistics, providing tools for analyzing the convergence properties of empirical measures and function classes. The classical Glivenko-Cantelli theorem asserts that the empirical distribution function converges uniformly to the true distribution function for independent and identically distributed (iid) samples. However, real-world data often exhibit dependence, necessitating an extension of this theory to dependent structures.\\

\noindent In this paper, we study the uniform convergence properties of real-valued empirical functions under $\alpha$-mixing and $\beta$-mixing conditions. These mixing conditions describe dependence structures commonly encountered in time series analysis, stochastic processes, and machine learning. Our goal is to establish sufficient conditions for the Glivenko-Cantelli property in these settings and to derive implications for statistical applications.\\

\noindent \textbf{Preliminaries and Definitions.}
Let $\lbrace X,(X_n)_{n\geq 1}\rbrace$ be a sequence of real-valued random variables (rrv's) defined on the same probability space $(\Omega,\mathcal{A},\mathbb{P})$ and associated to the same cumulative distribution function (cdf) $F$. Let $\mathcal{L}_2(\mathbb{P}_X)$ be the class of all measurable functions $f: \mathbb{R}\rightarrow \mathbb{R}$ with $\mathbb{E}f(X)^2<+\infty$. Let us define the functional empirical probability based on the $n$ first observations, $n\geq 1$, by

\begin{equation}\label{eq01}
\mathbb{P}_n(f)=\frac{1}{n}\sum_{i=1}^{n}f(X_i), \ f\in \mathcal{L}_2,
\end{equation}

\bigskip\noindent which will be centered at the mathematical expectations

\begin{equation*}
\mathbb{E}\mathbb{P}_n(f)=\mathbb{P}_X(f), \ f\in \mathcal{L}_2.
\end{equation*}

\bigskip\noindent In this paper, we are interested by functional Glivenko-Cantelli classes, that is classes $\mathcal{F}\subset \mathcal{L}_2$ for which we have

\begin{equation*}
	\sup_{f\in \mathcal{F}}|\mathbb{P}_n(f)-\mathbb{E}\mathbb{P}_n(f)|=\|\mathbb{P}_n-\mathbb{E}\mathbb{P}_n\|_{\mathcal{F}}\rightarrow 0,\ \text{almost surely, as } n\rightarrow +\infty.
\end{equation*}

\bigskip\noindent We might also consider a class $\mathcal{C}$ of measurable sets in $\mathbb{R}$ and define

\begin{equation}\label{eq02}
\mathbb{P}_n(C)=\frac{1}{n}\sum_{i=1}^{n}Card\lbrace X_i\in C\rbrace, \ C\in \mathcal{C},
\end{equation}

\bigskip\noindent with

\begin{equation*}
\mathbb{E}\mathbb{P}_n(C)=\mathbb{P}_X(C)=\mathbb{P}(X\in C).
\end{equation*}

\bigskip\noindent Taking $\mathcal{F}_c=\lbrace \mathbb{I}_{\rbrack-\infty,x\rbrack}, x\in \mathbb{R}\rbrace$ in Definition \eqref{eq01} or $\mathcal{C}_c=\lbrace \rbrack-\infty,x\rbrack, x\in \mathbb{R}\rbrace$ in Definition \eqref{eq02}, leads to the classical empirical function

\begin{equation*}
F_n(x)=\frac{1}{n}Card\lbrace i\in \lbrack 1,n\rbrack, X_i\leq x\rbrace, \ x\in \mathbb{R},
\end{equation*}

\bigskip\noindent which in turn gives the Glivenko-Cantelli law for independent and identically distributed (\textit{iid}) sequences $(X_n)_{n\geq 1}$ under the form :

\begin{equation}\label{eq03}
\sup_{x\in \mathbb{R}}|F_n(x)-F(x)|\rightarrow 0,\ \text{almost surely, as } n\rightarrow +\infty.
\end{equation}

\bigskip\noindent Such a result, also known as the fundamental theorem of statistics, is the frequentist paradigm (as opposed to the Bayesian paradigm) in statistics. In the form of \eqref{eq03}, the Glivenko-Cantelli law has gone through a large number of studies for a variety of types of dependence. It has also been used extensively in statistical theory, both for finding asymptotically efficient estimators and for related statistical tests based on the Donsker theorem, which we will not study here.\\

\noindent To give some examples, we cite the following results: Billingsley (1968) \cite{billingsley} showed the convergence in law on $D\lbrack 0,1\rbrack$ of the empirical process for $\phi$-mixing rrv's under the condition $\sum_{k>0}k^2\sqrt{\phi(k)}<+\infty$. Yoshihara (1975) \cite{yosh} obtained the same result for $\alpha$ mixing under the condition on the mixing coefficient $\alpha_n=O(n^{-a})$ with $a>3$. This result was first improved by Shao (1995) \cite{shao} by assuming only $a>2$, then by Shao and Yu (1996) \cite{shaoyu} by assuming $a>1+\sqrt{2}$. Rio (2017) \cite{rio} obtained an optimal condition $a>1$. Similar results are given by Shao and Yu (1996) \cite{shaoyu} for $\rho$-mixing rrv's, Doukhan \textit{et al.} (1995) \cite{doukh} for $\beta$-mixing rrv's. Next, Shao and Yu (1996) \cite{shaoyu} weakened the covariance condition $a>\frac{3+\sqrt{33}}{2}$. Louhichi (2000) \cite{louh} gave another proof and an improvement with $a>4$. Sangar\'e \textit{et al.} (2020) \cite{sanglotraore} gave a functional Glivenko-Cantelli theorem for $\phi$ mixing rrv's under the mixing coefficient condition $\phi_n=O(n^{-\frac{4}{1-\delta}})$ with $0<\delta<1$.

\bigskip\noindent However, functional versions do not seem to have been developed for dependent data, although they are much more interesting than restricting to the particular case of $\mathcal{F}_c$. Moreover, the functional version has the advantage of being linear in the sense that

\begin{equation*}
\forall (a,b)\in \mathbb{R}^2,\ \forall (f,g)\in (\mathcal{L}_1(\mathbb{P}_X))^2,\ \mathbb{P}_n(af+bg)=a\mathbb{P}_n(f)+b\mathbb{P}_n(g),
\end{equation*}

\bigskip\noindent which allows the use of more mathematical latitudes.\\

\noindent In this paper, we use general Glivenko-Cantelli classes established by Sangar\'e \textit{et al.} (2020) \cite{sanglotraore}, for real-valued empirical functions of dependent data and particularize them for $\alpha$-mixing and $\beta$-mixing random variables to obtain GC classes.\\
 
\noindent The main results are Theorem \ref{thg-gc-ef-01}, as a general GC theorem for arbitrary stationary sequences, Proposition \ref{prop01}, a Glivenko-Cantelli law for the real-valued empirical functions of stationary $\alpha$-mixing sequences, and Proposition \ref{prop02}, a Glivenko-Cantelli law for the real-valued empirical functions of stationary $\beta$-mixing sequences.\\

\noindent This functional approach requires the use of concentration numbers, which we define in the next section. We will also need to make a number of calls to dependency types such as $\alpha$-mixing and $\beta$-mixing, and other tools.\\

\noindent This leads us to the following organisation of the paper. In section \ref{sec04}, we recall the general strong law of large numbers of Sangar\'e and Lo (2015) \cite{sanglo}, which is crucial for the proofs of Sangar\'e \textit{et al.} (2020) \cite{sanglotraore} general GC classes used in our results, and to give a brief introduction to the notions of $\alpha$-mixing and $\beta$-mixing. But before that, section \ref{sec02} is devoted to entropy numbers and Vapnik-\v{C}ervonenkis classes. In section \ref{sec06}, we give our GC class regarding real-valued empirical functions for arbitrary stationary rv's sequences and particularize it for stationary $\alpha$-mixing and $\beta$-mixing rrv's sequences. The paper ends with Section \ref{sec07} which deals with conclusion and perspectives. 

\section{Entropy numbers and Vapnik-\v{C}ervonenkis} \label{sec02}

\noindent The following reminder about entropy numbers, which comes from combinatorial theory, might be perceived as making the paper heavier. However, we believe that it may help the reader who is not familiar with such techniques. Let $(E,\|\circ\|,\leq)$ be an ordered real normed space, which means that the order is compatible with the operations in the following sense

\begin{equation*}
\forall(\lambda,z)\in \mathbb{R}_+\vee\lbrace 0\rbrace\times E,\ \forall(x,y)\in E^2, (x<y)\rightarrow ((x+z<y+z) \ \text{ and } \ (\lambda x<\lambda y)).
\end{equation*}

\bigskip\noindent A bracket set of level $\varepsilon>0$ in $E$ is any set of the form

\begin{equation*}
	B(\ell_1, \ell_2,\varepsilon)=\lbrace x\in E, \ell_1\leq x\leq \ell_2\rbrace, \ \ell_1\leq \ell_2\in E \ \text{ and } \ \|\ell_2-\ell_1\|<\varepsilon.
\end{equation*}

\bigskip\noindent For any subset $F$ of $E$, the bracketing entropy number at level $\varepsilon>0$, denoted $N_{(F,\|\circ\|,\varepsilon)}$ is the minimum of the numbers $p\geq 1$ for which we have $p$ bracket sets (or simply $p$ brackets) at level $\varepsilon$ covering $F$, where the $\ell_i$'s do not necessarily belong to $F$. The bracketing entropy number is closely related to the Vapnick-\v{C}ervonenkis index (VC-index) of a Vapnick-\v{C}ervonenkis set or class (VC-set or VC-class).\\

\noindent To define a VC-set, we need to recall some definitions. A subset $D$ of $B\subset E$ is picked out by a subclass $\mathcal{C}$ of the power set $\mathcal{P}(E)$ from $B$ if and only if $D$ is element of $\lbrace B\cap C, C\in \mathcal{C}\rbrace$. Next, $B$ is shattered by $\mathcal{C}$ if and only if all subsets of $B$ are picked out by $\mathcal{C}$ from $B$. Finally, the class $\mathcal{C}$ is a VC-class if and only if there exists an integer $n\geq 1$ such that no set of cardinality $n$ is shattered by $\mathcal{C}$. The minimum of those numbers $n$ minus one is the index of that VC-class. The VC-class is the most quick way to bound bracketing entropy numbers $N_{(F,\|\circ\|,\varepsilon)}$, as stated in \cite{vdv} : If $\mathcal{C}$ is a VC-class of index $I(\mathcal{C})$, then

\begin{equation}
N_{(F,\|\circ\|,\varepsilon)}\leq KI(\mathcal{C})(4e)^{I(\mathcal{C})}(1/\varepsilon)^{rI(\mathcal{C}-1)},
\end{equation}

\bigskip\noindent where $K>0$ and $r>1$ are universal constants.\\

\noindent Finally, for a class $\mathcal{F}$ of real-valued functions $f : E\rightarrow \mathbb{R}$, we may define the bracketing entropy number associated with $\mathcal{F}$ to be the bracketing entropy number of the class $\mathcal{C}(\mathcal{F})$ of sub-graphs $f$ of elements $f\in \mathcal{F}$

\begin{equation*}
S_g(f)=\lbrace (x,t)\in E\times \mathbb{R}, f(x)>t\rbrace,
\end{equation*}

\bigskip\noindent in $E^\ast=E\times \mathbb{R}$ endowed with the product norm which is still a normed space. We denote by $N_{(F,\|\circ\|,\varepsilon)}$ to distinguish with the bracketing number using the norm of the functions $f\in \mathcal{F}$. Also, $\mathcal{F}$ is a VC-supg-class if and only if $\mathcal{C}(\mathcal{F})$ is a VC-class and $I(\mathcal{F})=I\mathcal{C}(\mathcal{F})$. We will use such entropy numbers to formulate VC-classes.\bigskip

\bigskip\noindent In the next section, we describe the main tools that we used in our results and provide a brief reminder of the concepts of $\alpha$ mixing and $\beta$ mixing.\\

\section{The Sangar\'e-Lo SLNN, The Sangar\'e-Lo-Traor\'e general functional Glivenko-Cantelli classes, $\alpha$-mixing and $\beta$-mixing concept} \label{sec04}

\noindent In this section, we have a look at the tools at the heart of our findings.

\subsection{General strong law of large numbers}
\noindent We begin with this useful result proved by \cite{sanglo}. 

\begin{theorem}(\cite{sanglo}). \label{lemmasanglo}
Let $X_1,X_2,\ldots$ be an arbitrary sequence of rv's, and let $(f_{i,n})_{i\geq 1}$ be a sequence of measurable functions such that $\mathbb{V}ar\lbrack f_{i,n}(X_i)\rbrack<+\infty$, for $i\geq 1$ and $n\geq 1$. If for some $\delta$, $0<\delta<3$

\begin{equation}
C_1=\sup_{n\geq 1}\sup_{q\geq 1}\mathbb{V}ar\left(\frac{1}{q^{(3-\delta)/4}}\sum_{i=1}^{q}f_{i,n}(X_i)\right)<+\infty \label{ghc_01}
\end{equation}

\bigskip\noindent and for some $0<\delta<3$,

\begin{equation}
C_2=\sup_{n\geq 0}\sup_{k\geq 0}\sup_{q:q^2+1\leq k\leq(q+1)^2 }\sup_{k\leq j\leq(q+1)^2 }\mathbb{V}ar\left(\frac{1}{q^{(3-\delta)/2}}\sum_{i=1}^{j-q^2+1}f_{q^2+i,n}(X_{q^2+i})\right)<+\infty \label{ghc_02}
\end{equation}

\bigskip\noindent hold, then

\begin{equation*}
\frac{1}{n}\sum_{i=1}^{n}\left(f_{i,n}(X_i)-\mathbb{E}(f_{i,n}(X_i))\right)\rightarrow 0, \text{almost surely, as} \ n\rightarrow +\infty.
\end{equation*}
\end{theorem}

\begin{remark}
We say that the sequence $X_{1}, X_{2},\cdots, X_{n}$ satisfies the \textbf{(GCIP)} whenever Conditions \eqref{ghc_01} and \eqref{ghc_02} hold, and we denote $\mathcal{C}$, the class of measurable functions for which \textbf{(GCIP)} holds.
\end{remark} 

\bigskip\noindent In the next subsection, we recall general functional Glivenko-Cantelli classes due to \cite{sanglotraore}. These results will be use in Section \ref{sec06} to find GC-classes for real-valued empirical function under $\alpha$-mixing and $\beta$-mixing conditions.\\

\subsection{General functional GC-classes}

\noindent Let us begin by giving a slight different version of Theorem (2.4.1 in \cite{vdv}, page 122), established by \cite{sanglotraore}.

\begin{theorem}(\cite{sanglotraore})\label{hgc_theo} Let $X, X_{1}, X_{2}, \ldots$ be an arbitrary stationary sequence of rrv's with common \textit{cdf} $F$.\\
	
	\noindent a) Let $\mathcal{C}$ be a class of measurable set such that \\
	
	\noindent (a1) for any $\varepsilon > 0$, $N_{[]}\left(\mathcal{C}, \left| \circ \right|, \varepsilon \right) < + \infty$,\\
	
	\noindent (a2) for any $C \in \mathcal{C}$, $\mathbb{P}_{n} \left(C\right) \longrightarrow \mathbb{P}_{X}\left(C\right)$, as $n \longrightarrow +\infty$.

	\bigskip \noindent Then $\mathcal{C}$ is GC-class, that is
	
	\begin{equation}\label{hgc_02}
		\lim_{n \to +\infty} \sup_{C \in \mathcal{C}}\left|\mathbb{P}_{n}\left(C\right) - \mathbb{P}_{X} \left( C \right) \right|=0.
	\end{equation}
	
	\bigskip \noindent b) Let $\mathcal{F}$ be a class of measurable set such that\\
	
	\noindent (b1) for any $\varepsilon > 0$, $N_{[]} \left( \mathcal{F}, \| \circ \|_{\mathcal{L}_{2}\left(\mathbb{P}_{X} \right)}, \varepsilon \right) < +\infty$,\\
	
	\noindent (b2) for any $f \in \mathcal{F}$, $\mathbb{P}_{n}\left(f\right) \longrightarrow \mathbb{P}_{X}\left(f\right)$, as $n \to + \infty$.
	
	\bigskip \noindent Then $\mathcal{F}$ is GC-class, that is
	
	\begin{equation}\label{hgc_03}
		\lim_{n \to +\infty} \sup_{f \in \mathcal{F}}\left| \mathbb{P}_{n} \left( f\right) - \mathbb{P}_{X} \left( f \right) \right|=0.
	\end{equation}
\end{theorem}

\begin{remark}
	If we apply Part (a) to $\mathcal{C}_c$, Condition (a1) can be dropped since it is VC-class of index 2. Indeed, for 
	$A=\{x_1, \ x_2, x_3\}$ with $x<x_2<x_3$. The subset $\{x_2\}$ cannot be picked out by $\mathcal{C}_c$ from $A$ since, for any $x\in \mathbb{R}$,  
	$]-\infty, x]\cap A$ will be of on the five sets : $\emptyset$ (for $x\leq x_1$), $\{x_1\}$ (for $x_1<x\leq x_2$), $\{x_1, x_2\}$ (for $x_2<x\leq x_3$) and $A$ (for $x>x_3$).\\
\end{remark}

\noindent Then finding GC-classes for $\mathcal{C}_c$ reduces to establishing Condition (b1). For now, we are going to focus on the application of GC-classes of the functional empirical process for $\mathcal{C}_c$, that is, on results as in Formula \eqref{eq03}. The focus will be on types of dependence, given it is known that $\mathcal{C}_c$ is GC-class for \textit{iid} data. Using the \textit{GCIP} conditions in Theorem \ref{lemmasanglo} leads to the applicable results as follows.

\begin{theorem}(\cite{sanglotraore})\label{hgc_theoApp} Let $X, X_{1}, X_{2}, \ldots$ be an arbitrary stationary sequence of rrv's with common \textit{cdf} $F$.\\
	
	\noindent a) Let $\mathcal{C}$ be a class of measurable set such that \\
	
	\noindent (a1) for any $\varepsilon > 0$, $N_{[]}\left(\mathcal{C}, \left| \circ \right|, \varepsilon \right) < + \infty$,\\
	
	\noindent (a2) for any $C \in \mathcal{C}$, the following conditions hold : for some $0<\delta<3$,
	
	\begin{equation}
		\sup_{q\geq 1}\dfrac{1}{q^{(3-\delta )/2}}\sum_{i,j=1}^{q}\left[ 
		\mathbb{P(}X_{i}\in C,X_{j}\in C)-\mathbb{P(}X_{i}\in C)\mathbb{P}(
		X_{j}\in C)\right] <+\infty. \label{ghcC_01}
	\end{equation}
	
	\bigskip\noindent and 
	
	\begin{equation}
		\sup_{k\geq 1}\sup_{q\text{ : }q^{2}+1\leq k\leq
			(q+1)^{2}}\sup_{k\leq j\leq (q+1)^{2}}\dfrac{1}{q^{(3-\delta )}}%
		\sum_{i,\ell=1}^{j-q^{2}+1}\left[ 
		\begin{array}{c}
			\mathbb{P(}X_{_{q^{2}+i}}\in C,X_{_{q^{2}+\ell}}\in C) \\ 
			-\mathbb{P(}X_{_{q^{2}+i}}\in C)\mathbb{P(}Y_{_{q^{2}+\ell}}\in C)%
		\end{array}
		\right] <+\infty . \label{ghcC_02}
	\end{equation}
	\bigskip \noindent Then $\mathcal{C}$ is GC-class, that is
	\begin{equation}\label{hgc_02}
		\lim_{n \to +\infty} \sup_{C \in \mathcal{C}}\left|\mathbb{P}_{n}\left(C\right) - \mathbb{P}_{X} \left( C \right) \right|=0.
	\end{equation}
	
	\bigskip \noindent b) Let $\mathcal{F}$ be a class of measurable set such that\\
	
	\noindent (b1) for any $\varepsilon > 0$, $N_{[]} \left( \mathcal{F}, \| \circ \|_{\mathcal{L}_{2}\left(\mathbb{P}_{X} \right)}, \varepsilon \right) < +\infty$,\\
	
	\noindent (b2) for any $f \in \mathcal{F}$, the following conditions hold : for some $0<\delta<3$,
	
	\begin{equation}
		\sup_{q\geq 1}\mathbb{V}ar\left(\frac{1}{q^{(3-\delta)/4}}\sum_{i=1}^{q}f(X_i)\right)<+\infty \label{ghcf_01} 
	\end{equation}
	
	\bigskip\noindent and 
	
	\begin{equation}
		\sup_{k\geq 0}\sup_{q:q^2+1\leq k\leq(q+1)^2 }\sup_{k\leq j\leq(q+1)^2 }\mathbb{V}ar\left(\frac{1}{q^{(3-\delta)/2}}\sum_{i=1}^{j-q^2+1}f(X_{q^2+i})\right)<+\infty. \label{ghcf_02}
	\end{equation}
	
	\bigskip \noindent Then $\mathcal{F}$ is GC-class, that is
	
	\begin{equation}\label{hgc_03}
		\lim_{n \to +\infty} \sup_{f \in \mathcal{F}}\left| \mathbb{P}_{n} \left( f\right) - \mathbb{P}_{X} \left( f \right) \right|=0.
	\end{equation}
\end{theorem}

\bigskip \noindent Since we are also going to provide results for sequences verifying $\alpha$-mixing and $\beta$-mixing conditions, we make a brief summary of these notions.\\

\subsection{$\alpha$-mixing}\label{alphamixing}

\noindent Let us have a brief recall on $\phi$-mixing. Let ($\Omega ,\mathcal{A},\mathbb{P}$) be a probability space and $\mathcal{A}_1$, $\mathcal{A}_2$ two sub $\sigma -$algebras of $\mathcal{A}$. The $\alpha $-mixing coefficient is given by 

\begin{equation*}
\alpha (\mathcal{A}_1,\mathcal{A}_2)=\sup \left\{ \left\vert \mathbb{P}(A)\mathbb{P}(B)-\mathbb{P}(A\cap B)\right\vert, \
 A \in \mathcal{A}_1, \  B \in \mathcal{A}_2\right\}
\end{equation*}
and
\begin{equation*}
	0\leq\alpha (\mathcal{A}_1,\mathcal{A}_2)\leq 1/4.
\end{equation*}
\noindent We recall that we have (see \cite{paul}) the following result on the covariance inequality for $\alpha$-mixing random sequences. We present here all the material used in the proof of this result in a detailed writing, which makes it more understandable for a wide audience. 

\begin{theorem}(Covariance inequality).\label{theo1}
	Let $X$ and $Y$ be mesurable random variables with respect to $A$ and $B$ respectively, We have 
	
	\begin{equation*}
		|\mathbb{C}ov(X,Y)| \leq 8 \alpha^\frac{1}{r}(\mathcal{A},\mathcal{B})||X||_p||Y||_q \ for \ p, \ q,\ r \geq 1 \text{ and } \frac{1}{p}+ \frac{1}{q}+\frac{1}{r}=1.
	\end{equation*}
\end{theorem}

\noindent To prove the above theorem, we need the following Lemma.

\begin{lemma}\label{lem1} Under the same assumptions, we have the following result
	
	\begin{equation*}
		|\mathbb{C}ov(X,Y)| \leq 4 \alpha(\mathcal{A} ,\mathcal{B})||X||_\infty||Y||_{\infty}. 
	\end{equation*}
\end{lemma}

\begin{proof}(of Lemma \ref{lem1}) Let $A = sign(\mathbb{E}(X \mid \mathcal{B}) - \mathbb{E}(X))$ and $B = sign(\mathbb{E}(Y\mid \mathcal{A}) - \mathbb{E}(Y))$. By the conditional expectation property, we have
	
	\begin{eqnarray*}
		|\mathbb{C}ov(X,Y)| &=& \mid \mathbb{E}(XY)-\mathbb{E}(X) \mathbb{E}(Y)\mid\\
		&=&\mid \mathbb{E}(X\mathbb{E}(Y\mid\mathcal{A})-\mathbb{E}(Y))\mid\\
		&\leq& ||X||_\infty \mathbb{E}\mid\mathbb{E}(Y\mid\mathcal{A})-\mathbb{E}(Y)\mid\\
		&\leq& ||X||_\infty \mathbb{E}|B[(Y\mid\mathcal{A})-\mathbb{E}(Y)]\mid\\
		&\leq& ||X||_\infty \mid \mathbb{E}(BY)-\mathbb{E}(B)\mathbb{E}(Y)\mid
	\end{eqnarray*}
	
	\noindent and since
	
	\begin{eqnarray*}
		\mid\mathbb{E}(BY)-\mathbb{E}(B)\mathbb{E}(Y)|&=&|\mathbb{E}[\mathbb{E}(YB \mid \mathcal{B})-\mathbb{E}(B)\mathbb{E}(Y)]\mid\\
		&=&\mid \mathbb{E}[Y\mathbb{E}(B \mid \mathcal{B})-\mathbb{E}(B)]\mid\\
		&\leq& ||Y||_\infty \mathbb{E}|\mathbb{E}(B\mid\mathcal{B})-\mathbb{E}(B)|\\
		&\leq& ||X||_\infty \mathbb{E}\mid A[(B \mid \mathcal{B})-\mathbb{E}(B)]\mid\\
		&\leq& ||Y||_\infty |\mathbb{E}(AB)-\mathbb{E}(B)\mathbb{E}(A)|,
	\end{eqnarray*}
	
	\noindent therefore
	
	$$|\mathbb{C}ov(X,Y)| \leq ||X||_\infty ||Y||_\infty |\mathbb{E}(AB)-\mathbb{E}(A)\mathbb{E}(B)|.$$
	
	\bigskip\noindent Let now $A^+=\{a=1\}, A^-= \{a= -1\}, B^+ =\{b =1\}, B^-=\{b = -1\}$, then 
	
	\begin{eqnarray*}
		|\mathbb{E}(AB)-\mathbb{E}(A)\mathbb{E}(B)|&=&|[\mathbb{P}(A^+\cap B^+)-\mathbb{P}(A^+)\mathbb{P}(B^+)]\\
		&&+[\mathbb{P}(A^-\cap B^-)-\mathbb{P}(A^-)\mathbb{P}(B^-)]\\
		&&+[\mathbb{P}(A^+\cap B^-)-\mathbb{P}(A^+)\mathbb{P}(B^-)]\\
		&&+[\mathbb{P}(A^-\cap B^+)-\mathbb{P}(A^-)\mathbb{P}(B^+)]|.
	\end{eqnarray*}
	Hence 
	\begin{equation*}
		|\mathbb{C}ov(X,Y)| \leq 4\alpha(\mathcal{A},\mathcal{B})||X||_\infty ||Y||_\infty .
	\end{equation*}
\end{proof}

\noindent We are now ready to prove the theorem \ref{theo1}.

\begin{proof} (of Theorem \ref{theo1})\\
	$$||X||_p< \infty, \ ||Y||_p< \infty, \ 1<p<\infty.$$
	
	\bigskip\noindent Define $\overline{X}=X  \mathbb{I}_{\{|X|\leq a\}}$,    $\underline{X}=X  \mathbb{I}_{\{|X|> a\}}$ and write $X=\overline{X} + \underline{X}.$ Thus, by the Lemma \ref{lem1}, 
	
	$$\mathbb{C}ov(X,Y)=|\mathbb{C}ov(\overline{X},Y)+ \mathbb{C}ov(\underline{X}, Y)| \leq 4\alpha(\mathcal{A},\mathcal{B}) a ||Y||_\infty +2 ||Y||_\infty \mathbb{E}| \underline{X}|.$$
	
	\bigskip\noindent Now Markov's inequality leads to $\mathbb{E}|\underline{X}|\leq 2 a \frac{\mathbb{E}|X|^p}{a^p}.$ Set $\frac{\mathbb{E}|X|^p}{a^p}=\alpha(\mathcal{A},\mathcal{B})$ then 
	
	\begin{equation*}
		|\mathbb{C}ov(X,Y)| \leq 6 \alpha^{1-1/p}(\mathcal{A},\mathcal{B})||Y||_\infty ||X||_\infty.
	\end{equation*}
	
	\bigskip\noindent Let $||X||_p<\infty \text{ and } ||Y||_q <\infty, \text{ with } \frac{1}{p}+\frac{1}{q}<1.$
	
	\bigskip\noindent Let us define $\overline{Y}=Y \mathbb{I}_{\{|Y|\leq b\}}$,    $\underline{Y}=Y \mathbb{I}_{\{|Y|> b\}}$  and write $Y=\overline{Y} + \underline{Y},$ hence analogously  
	
	\begin{equation*}
		|\mathbb{C}ov(Y,X)|\leq 6 \alpha^{1-1/p}(\mathcal{A},\mathcal{B})b ||X||_p + 2b ||\underline{Y}||.
	\end{equation*}
	
	\bigskip\noindent By setting $\frac{\mathbb{E}|Y|^q}{a^q}=\alpha^{1-1/p}(A,B)$, it yields
	\begin{equation*}
		|\mathbb{C}ov(X,Y)|\leq 8 \alpha^{1/r}(\mathcal{A},\mathcal{B})||Y||_q||X||_p.
	\end{equation*}
\end{proof}
\noindent We define the strong mixing coefficient $\alpha(n)$ by

\begin{equation*}
	\alpha (n)=\sup \left\{ \alpha \left(\mathcal{F}_{1}^{k},\mathcal{F}_{n+k}^{\infty }\right), \ k\in \mathbb{N}^{\ast }\right\} 
\end{equation*}

\bigskip\noindent where $\mathcal{F}_{j}^{\ell }$ is the $\sigma -$algebra generated by the variables $(X_{i}$, $j\leq i\leq \ell )$. We say that  $(X_{n})_{n\geq 1}$ is $\alpha$-mixing if $\alpha (n)\rightarrow 0$ as $n\rightarrow +\infty$. For some further clarification on this point, the reader may have a quick look at \cite{paul} and \cite{rio}.\\

\subsection{$\beta$-mixing} 
Let $\mathcal{A}_{1,n}$ be the $\sigma $-algebra generated by the finite subset $X_{1},\cdots,X_{n}$ and $\mathcal{A}_{2,n}$ that generated by the infinite subset  $X_{n+k+1},X_{n+k+2},\ldots $ We define the mixing of type $\beta $ by means of the quantity

\begin{equation*}
	\beta (n,k)=\beta(\mathcal{A}_{1,n},\mathcal{A}_{2,n})=\mathbb{E}\left[ \text{ess}-\sup \left\{\mathbb{P}(B|\mathcal{A}_{1,n})-%
		\mathbb{P}(B),\ B\in \mathcal{A}_{2,n}\right\}\right]
\end{equation*}
and
\begin{equation*}
	0\leq \beta(\mathcal{A}_{1,n},\mathcal{A}_{2,n}) \leq 1.
\end{equation*}

\bigskip\noindent This coefficient is called the absolute  regularity, or $\beta$-mixing coefficient, it may be rewritten as 

\begin{equation*}
	\beta (n,k)=\sup_{A_i\in\mathcal{A}_{1,n},\ B_i\in\mathcal{A}_{2,n}} \left\{ \frac{1}{2}\sum_{i=1}^{I}\sum_{j=1}^{J}|\mathbb{P}(A_i\cap B_j)-\mathbb{P}(A_i)\mathbb{P}(B_j)|\right\}.
\end{equation*}

\bigskip\noindent For more details on this question, the reader may have a quick look at \cite{paul}, \cite{rio} and \cite{Bradley}.

\bigskip\noindent In the next section, we give our results for weakly dependent sequences, under assumptions of stationary $\alpha$-mixing and $\beta$-mixing sequences for real-valued empirical functions.

\section{Results and Discussion} \label{sec06}
\noindent We begin with the main result.
\subsection{GC-classes for real-valued empirical functions}

\noindent We focus on $\mathcal{F}_c$ and define the empirical function

\begin{equation*}
F_{n}(x)=\frac{1}{n}\sum_{i=1}^{n} \mathbb{I}_{]-\infty,\ x]}(X_i), \ x \in \mathbb{R},
\end{equation*}

\noindent and its expectation

\begin{equation*}
F(x)=\mathbb{E}\left(\mathbb{I}_{]-\infty, \ x]}(X_i)\right), \ x \in \mathbb{R}.
\end{equation*}

\bigskip\noindent Now, by applying the Part (a) of Theorem (\ref{hgc_theoApp}, page \pageref{hgc_theoApp}) for $C_x=]-\infty, x]$, $x \in \mathbb{R}$, we get

\begin{theorem} \label{thg-gc-ef-01} Let $X$, $X_{1}$ ,$X_{2}$, $\ldots$  be an arbitrary stationary and square integrable sequence of \textit{rrv}'s, defined on a probability space $(\Omega, \mathcal{A}, \mathbb{P})$. We have 
	
\begin{equation*}
\sup_{x\in \mathbb{R}}|F_n(x)-F(x)|\rightarrow 0, \ \text{almost surely, as} \ n\rightarrow +\infty
\end{equation*}

	\noindent whenever the following general conditions hold : for some $\delta \in ]0;3[$
	
	\begin{equation}
		\sup_{q\geq 1}\frac{1}{q^{(3-\delta )/2}}\left[ \sum_{i=1}^{q}\mathbb{V}ar\left( \mathbb{I}_{\left\{ X_{i}\leq
			x\right\} }\right)+ \sum_{(i\neq j)=1}^{q}\mathbb{C}ov\left(\mathbb{I}_{\left\{ X_{i}\leq x\right\} },\mathbb{I}_{\left\{ X_{j}\leq x\right\} }\right)\right] <+\infty \label{gcep1}
	\end{equation}
	and
	\begin{equation} \label{gcep2}
		\sup_{q\geq 1}\frac{1}{q^{(3-\delta )}}%
		\left[\sum_{i=q^{2}+1}^{(q+1)^{2}}\mathbb{V}ar\left(
		\mathbb{I}_{\left\{ X_{i}\leq x\right\} }\right)+\sum_{(i\neq j)=q^{2}+1}^{(q+1)^{2}}\mathbb{C}ov\left(
		\mathbb{I}_{\left\{ X_{i}\leq x\right\} },\mathbb{I}_{\left\{ X_{j}\leq x\right\} }\right)\right]<+\infty .
	\end{equation}
\end{theorem}

\begin{remark}
	Since in the expressions in \eqref{gcep1} and \eqref{gcep2}, the covariances are zero whenever the variables are independent and identically distributed (iid), the iid case remains valid without any further condition, which is the classical Glivenko-Cantelli theorem.
\end{remark}

\begin{proof}
	We apply the Part (a) of Theorem (\ref{hgc_theoApp}, page \pageref{hgc_theoApp}) for $C_x=]-\infty, x]$, $x \in \mathbb{R}$ and verify the conditions \eqref{ghc_01} and \eqref{ghc_02} in Theorem (\ref{lemmasanglo}, page \pageref{lemmasanglo}). The conditions become
	
	\begin{eqnarray*}
		\sup_{q\geq 1}\mathbb{V}ar\left( \frac{1}{q^{(3-\delta )/4}}%
		\sum_{i=1}^{q}\mathbb{I}_{\left\{ X_{i}\leq x\right\} }\right) &=&\sup_{q\geq 1}\frac{1}{q^{(3-\delta )/2}}\mathbb{V}ar\left( \sum_{i=1}^{q}\mathbb{I}_{\left\{ X_{i}\leq
			x\right\} }\right)  \notag\\
		&=&\sup_{q\geq 1}\frac{1}{q^{(3-\delta )/2}}\sum_{i,j=1}^{q} \mathbb{C}ov\left(\mathbb{I}_{\left\{ X_{i}\leq x\right\} },\mathbb{I}_{\left\{ X_{j}\leq x\right\} }\right) \\
		\\
		&=&\sup_{q\geq 1}\frac{1}{q^{(3-\delta )/2}}\left[ \sum_{i=1}^{q}\mathbb{V}ar\left( \mathbb{I}_{\left\{ X_{i}\leq
			x\right\} }\right)+ \sum_{(i\neq j)=1}^{q}\mathbb{C}ov\left(\mathbb{I}_{\left\{ X_{i}\leq x\right\} },\mathbb{I}_{\left\{ X_{j}\leq x\right\} }\right)\right] \\
		&<&+\infty
	\end{eqnarray*}
	
	\noindent and
	
	\begin{eqnarray*}
		&&\sup_{k\geq 1}\sup_{q\text{ : }q^{2}+1\leq k\leq (q+1)^{2}}\sup_{k\leq
			j\leq (q+1)^{2}}\mathbb{V}ar\left( \dfrac{1}{q^{(3-\delta )/2}}%
		\sum_{i=1}^{j-q^{2}+1}\mathbb{I}_{\left\{ X_{q^{2}+i}\leq x\right\} }\right) 
		\notag \\
		&=&\sup_{k\geq 1}\sup_{q\text{ : }q^{2}+1\leq k\leq (q+1)^{2}}\sup_{k\leq
			j\leq (q+1)^{2}}\frac{1}{q^{(3-\delta )}}\mathbb{V}ar\left(
		\sum_{i=1}^{j-q^{2}+1}\mathbb{I}_{\left\{ X_{q^{2}+i}\leq x\right\} }\right) 
		\notag \\
		&=&\sup_{q\geq 1}\frac{1}{q^{(3-\delta )}}%
		\mathbb{V}ar\left(
		\sum_{i=q^{2}+1}^{(q+1)^{2}}\mathbb{I}_{\left\{ X_{i}\leq x\right\} }\right) \\
		&=&\sup_{q\geq 1}\frac{1}{q^{(3-\delta )}}%
		\sum_{i,j=q^{2}+1}^{(q+1)^{2}}\mathbb{C}ov\left(
		\mathbb{I}_{\left\{ X_{i}\leq x\right\} },\mathbb{I}_{\left\{ X_{j}\leq x\right\} }\right) \\
		&=&\sup_{q\geq 1}\frac{1}{q^{(3-\delta )}}%
		\left[\sum_{i=q^{2}+1}^{(q+1)^{2}}\mathbb{V}ar\left(
		\mathbb{I}_{\left\{ X_{i}\leq x\right\} }\right)+\sum_{(i\neq j)=q^{2}+1}^{(q+1)^{2}}\mathbb{C}ov\left(
		\mathbb{I}_{\left\{ X_{i}\leq x\right\} },\mathbb{I}_{\left\{ X_{j}\leq x\right\} }\right)\right] \\
		&<&+\infty. 
	\end{eqnarray*}
\end{proof}

\begin{remark}
	Note that for sequences of stationary random variables, the condition \eqref{gcep1} implies that of \eqref{gcep2}. 
\end{remark}

\noindent In the sequel, we use the Theorem (\ref{thg-gc-ef-01} page \pageref{thg-gc-ef-01}) to give applications for stationary $\alpha$-mixing and $\beta$-mixing sequences. 

\subsection{$\alpha $-mixing case}

\noindent We already made a brief recall on $\alpha$-mixing in Subsection \ref{alphamixing} in Section \ref{sec04}. We are now going to provide applications to it.\\

\begin{proposition}\label{prop01}
	Suppose that $X_j's$ form a $\alpha-$mixing stationary and square integrable sequence of random variables with mixing condition $\alpha$. Then
	
	\begin{equation*}
		\sup_{x \in \mathbb{R}} |F_n(x)-F(x)|\rightarrow 0,  \text{ almost surely, as}\  n\rightarrow + \infty
	\end{equation*}
	
	\noindent holds whenever we have, for some $\delta \in ]0;1[$,
	
	\begin{equation*}
		\alpha(n)=O\left(n^{-\left(\frac{1+\delta}{1-\delta}\right)} \right). 
	\end{equation*}
\end{proposition}

\begin{proof}
We need to verify the condition \eqref{gcep1} in Theorem (\ref{thg-gc-ef-01}, page \pageref{thg-gc-ef-01}) 	
	\begin{eqnarray*}
		\sup_{q \geq 1} \frac{1}{q^{(3 - \delta)/2}} \mathbb{V}ar\left(\sum_{i=1}^q \mathbb{I}_{\left\{X_i \leq x\right\}}\right) &=&\sup_{q\geq 1}\frac{1}{q^{(3-\delta )/2}}\left[ \sum_{i=1}^{q}\mathbb{V}ar\left( \mathbb{I}_{\left\{ X_{i}\leq
			x\right\} }\right)+ \sum_{(i\neq j)=1}^{q}\mathbb{C}ov\left(\mathbb{I}_{\left\{ X_{i}\leq x\right\} },\mathbb{I}_{\left\{ X_{j}\leq x\right\} }\right)\right] \\
		&=&\sup_{q\geq 1}\frac{1}{q^{(3-\delta )/2}}\left[ \sum_{i=1}^{q}\mathbb{V}ar\left( \mathbb{I}_{\left\{ X_{i}\leq
			x\right\} }\right)+ 2\sum_{1\leq i< j\leq q}\mathbb{C}ov\left(\mathbb{I}_{\left\{ X_{i}\leq x\right\} },\mathbb{I}_{\left\{ X_{j}\leq x\right\} }\right)\right] \\
		& \leq& \frac{q}{q^{(3 - \delta)/2}}+\frac{16}{q^{(3 - \delta)/2}} \sum_{1 \leq i \leq q-1 ,\ j = i+h,\ h< n} \alpha(|i-j|)\\
		&&+ \frac{16}{q^{(3 - \delta)/2}} \sum_{1 \leq i \leq q-1 ,\ j = i+h,\ h\geq n} \alpha(|i-j|) \text{ by the Theorem \ref{theo1}}\\	
		& \leq& \frac{q}{q^{(3 - \delta)/2}}+\frac{8nq}{q^{(3 - \delta)/2}}\times\frac{1}{4}+ \frac{8q^2-8nq-q}{q^{(3 - \delta)/2}} \alpha(n)\\
		& \sim& \frac{2nq}{q^{(3 - \delta)/2}}+ \frac{8q^2}{q^{(3 - \delta)/2}} \alpha(n)<+\infty.
	\end{eqnarray*}
	
	\noindent Finally, the condition reduces to
	
	\begin{equation*}
		n=\left[q^{\left(1- \delta\right)/2} \right] \text{ and }  \alpha(n)  =O\left(q^{\frac{1+\delta}{2}} \right) =O\left(n^{-\left(\frac{1+\delta}{1-\delta}\right)} \right).
	\end{equation*}
\end{proof}  

\begin{remark}
	 In the case of $\alpha$-mixing sequences, we do not need whole function $\alpha$. Instead we may fix $x \in \mathbb{R}$ and consider the modulus $\alpha(x,n)$ related to the sequence $\mathbb{I}_{]-\infty,x]}(X_i)$'s. It is clear that 
\begin{equation*}
	\forall n\geq 1, \ \forall x \in \mathbb{R}, \ \alpha(x,n)\leq \alpha(n).
\end{equation*}

	\bigskip\noindent We have the Glivenko-Cantelli theorem if and only if
	
	\begin{equation*}
		\forall x \in \mathbb{R}, \ \alpha(n)  =O\left(n^{-\left(\frac{1+\delta}{1-\delta}\right)} \right).
	\end{equation*}
	
	\bigskip\noindent The same can be done for any particular class $\mathcal{C}$ by using the sequences $\mathbb{I}_{C}(X_i)$'s and the $\alpha(C,n)$ mixing modulus, and get the condition
	
	\begin{equation*}
		\forall x \in \mathcal{C} , \ \alpha(C,n)=O\left(n^{-\left(\frac{1+\delta}{1-\delta}\right)} \right).
	\end{equation*}
\end{remark}

\subsection{$\beta $-mixing case}
We consider here a stationary process $(X_t)_{t \in \mathbb{Z}}$ satisfying a $\beta$-mixing condition. The $\beta$-mixing coefficient is defined as

\begin{equation*}
	\beta(n) = \sup_{A \in \mathcal{F}_{t}^{-\infty},\ B \in \mathcal{F}_{t+n}^{\infty}} \left| \mathbb{P}(A \cap B) - \mathbb{P}(A) \mathbb{P}(B) \right|,
\end{equation*}

\bigskip\noindent where $\mathcal{F}_{t}^{-\infty} = \sigma(X_\ell : \ell \leq t)$ and $\mathcal{F}_{t+n}^{\infty} = \sigma(X_\ell : \ell \geq t+n)$.

\bigskip\noindent We begin with giving this useful result which formulates the covariance inequality for stationary $\beta$-mixing sequences.

\begin{lemma}\label{theo02}
	For any stationary beta-mixing process $(X_t)_{t \in \mathbb{Z}}$ and for all $t, s \in \mathbb{Z}$, we have
	
	\begin{equation*}
		|\mathbb{C}ov(X_t, X_s)| \leq 2\beta(|t - s|) \|X_t\|_\infty \|X_s\|_\infty,
	\end{equation*}
	
	\bigskip\noindent  where $\|X_t\|_\infty = \sup_{\omega \in \Omega} |X_t(\omega)|$ is the supremum norm of $X_t$.
\end{lemma}

\begin{proof}
	The covariance of $X_t$ and $X_s$ is given by
	
	\begin{equation*}
		\mathbb{C}ov(X_t, X_s) = \mathbb{E}[X_t X_s] - \mathbb{E}[X_t] \mathbb{E}[X_s].
	\end{equation*}
	
	\bigskip\noindent We focus on the first term, $\mathbb{E}[X_t X_s]$. By introducing the $\sigma$-algebra $\mathcal{F}_t = \sigma(X_\ell : \ell \leq t)$, we can write
	
	\begin{equation*}
		\mathbb{E}[X_t X_s] = \mathbb{E}[X_t \mathbb{E}[X_s | \mathcal{F}_t]].
	\end{equation*}
	
	\noindent Thus, the covariance becomes
	
	\begin{equation*}
		\mathbb{C}ov(X_t, X_s) = \mathbb{E}[X_t (\mathbb{E}[X_s | \mathcal{F}_t] - \mathbb{E}[X_s])].
	\end{equation*}
	
	\bigskip\noindent The definition of the coefficient $\beta(n)$ implies that for any random variable $X_s \in \mathcal{F}_{t+n}^{\infty}$, we have the bound
	
	\begin{equation*}
		|\mathbb{E}[X_s | \mathcal{F}_t] - \mathbb{E}[X_s]| \leq 2\beta(n) \|X_s\|_\infty,
	\end{equation*}
	
	\noindent where $n = |t - s|$.\\
	
	\noindent By substituting into the previous expression, we obtain
	
	\begin{equation*}
		|\mathbb{C}ov(X_t, X_s)| \leq \|X_t\|_\infty \cdot \mathbb{E}[|\mathbb{E}[X_s | \mathcal{F}_t] - \mathbb{E}[X_s]|].
	\end{equation*}
	
	\bigskip\noindent   Since $|\mathbb{E}[X_s | \mathcal{F}_t] - \mathbb{E}[X_s]| \leq 2\beta(|t - s|)\|X_s\|_\infty$, it follows that
	
	\begin{equation*}
		|\mathbb{C}ov(X_t, X_s)| \leq 2\beta(|t - s|) \|X_t\|_\infty \|X_s\|_\infty.
	\end{equation*}
	
	\bigskip\noindent  Thus, the proof of Theorem \ref{theo02}  is complete.
\end{proof}

\bigskip\noindent The next result states the uniform convergence for empirical functions of $\beta-$mixing stationary sequences.\\

\begin{proposition}\label{prop02}
	Suppose that $X_j's$ form a $\beta-$mixing stationary and square integrable sequence of random variables with mixing condition $\beta$. Then
	
	\begin{equation*}
		\sup_{x \in R} |F_n(x)-F(x)|\rightarrow 0, \text{ almost surely, as } n\rightarrow + \infty
	\end{equation*}
	
	\noindent holds whenever we have, for some $\delta \in ]0;1[$,
	
	\begin{equation*}
		\beta(n)=O\left(n^{-\left(\frac{1+\delta}{1-\delta}\right)} \right). 
	\end{equation*}
\end{proposition}

\begin{proof}
	We need to verify the condition \eqref{gcep1} in Theorem (\ref{thg-gc-ef-01}, page \pageref{thg-gc-ef-01}) 	
	\begin{eqnarray*}
		\sup_{q \geq 1} \frac{1}{q^{(3 - \delta)/2}} \mathbb{V}ar\left(\sum_{i=1}^q \mathbb{I}_{\left\{X_i \leq x\right\}}\right) &=&\sup_{q\geq 1}\frac{1}{q^{(3-\delta )/2}}\left[ \sum_{i=1}^{q}\mathbb{V}ar\left( \mathbb{I}_{\left\{ X_{i}\leq
			x\right\} }\right)+ \sum_{(i\neq j)=1}^{q}\mathbb{C}ov\left(\mathbb{I}_{\left\{ X_{i}\leq x\right\} },\mathbb{I}_{\left\{ X_{j}\leq x\right\} }\right)\right] \\
		&=&\sup_{q\geq 1}\frac{1}{q^{(3-\delta )/2}}\left[ \sum_{i=1}^{q}\mathbb{V}ar\left( \mathbb{I}_{\left\{ X_{i}\leq
			x\right\} }\right)+ 2\sum_{1\leq i< j\leq q}\mathbb{C}ov\left(\mathbb{I}_{\left\{ X_{i}\leq x\right\} },\mathbb{I}_{\left\{ X_{j}\leq x\right\} }\right)\right] \\
		& \leq& \frac{q}{q^{(3 - \delta)/2}}+\frac{4}{q^{(3 - \delta)/2}} \sum_{1 \leq i \leq q-1,\ j = i+h,\ h< n} \beta(|i-j|)\\
		&&+ \frac{4}{q^{(3 - \delta)/2}} \sum_{1 \leq i \leq q-1 ,\ j = i+h,\ h\geq n} \beta(|i-j|)\text{ by the Lemma \ref{theo02}}\\	
		& \leq& \frac{q}{q^{(3 - \delta)/2}}+\frac{2nq}{q^{(3 - \delta)/2}}+ \frac{2q^2-2nq-q}{q^{(3 - \delta)/2}} \beta(n)\\
		& \sim& \frac{2nq}{q^{(3 - \delta)/2}}+ \frac{2q^2}{q^{(3 - \delta)/2}} \beta(n)<+\infty.
	\end{eqnarray*}
	
	\noindent Finally, the condition reduces to
	
	\begin{equation*}
		n=\left[q^{\left(1- \delta\right)/2} \right] \text{ and }  \beta(n) =O\left(q^{\frac{1+\delta}{2}} \right) =O\left(n^{-\left(\frac{1+\delta}{1-\delta}\right)} \right).
	\end{equation*}
\end{proof} 

\begin{remark}
	In the case of $\beta$-mixing sequences too, we do not need whole function $\beta$. Instead we may fix $x \in \mathbb{R}$ and consider the modulus $\beta(x,n)$ related to the sequence $\mathbb{I}_{]-\infty,x]}(X_i)$'s. It is clear that 
	
	\begin{equation*}
		\forall n\geq 1, \ \forall x \in \mathbb{R}, \ \beta(x,n)\leq \beta(n). 
	\end{equation*}

	\bigskip\noindent We have the Glivenko-Cantelli theorem if and only if
	
	\begin{equation*}
		\forall x \in \mathbb{R}, \ \beta(n)  =O\left(n^{-\left(\frac{1+\delta}{1-\delta}\right)} \right).
	\end{equation*}
	
	\bigskip\noindent The same can be done for any particular class $\mathcal{C}$ by using the sequences $\mathbb{I}_{C}(X_i)$'s and the $\beta(C,n)$ mixing modulus, and get the condition
	
	\begin{equation*}
		\forall x \in \mathcal{C} , \ \beta(C,n)=O\left(n^{-\left(\frac{1+\delta}{1-\delta}\right)} \right).
	\end{equation*}
\end{remark}

\noindent Several authors have previously established Glivenko-Cantelli-type results under dependent sequences, particularly for mixing processes:

\bigskip\noindent (a) \cite{billingsley} – $\phi$-Mixing Condition

\begin{itemize}
	\item	Result: Showed weak convergence of the empirical process for $\phi$-mixing sequences.
	
	\item	Assumption: Required the summability condition 
	$\sum_{k>0} k^2 \sqrt{\phi(k)} < +\infty$ where $\phi(k)$ measures dependence at lag $k$.
	
	\item	Limitation: The $\phi$-mixing condition is often too strong and restrictive in practical applications.
\end{itemize}

\bigskip\noindent (b) \cite{yosh} – $\alpha$-Mixing with Strong Decay Rate
\begin{itemize}
	\item Result: Proved weak convergence of the empirical process for $\alpha$-mixing sequences.
	
	\item Assumption: Required the mixing coefficient decay rate $\alpha(n) = O(n^{-a})$ with $a > 3. $
	
	\item Limitation: The assumption $a>3$ was too restrictive, excluding many common dependent sequences found in applications (e.g., financial data, time series).
\end{itemize}

\bigskip\noindent (c) \cite{shao} – Improved $\alpha$-Mixing Rate
\begin{itemize}
	\item Result: Improved Yoshihara's bound by relaxing the decay rate to $\alpha(n) = O(n^{-a})$  with $a > 2.$ 
	
	\item \cite{shaoyu} further weakened the condition to $a > 1 + \sqrt{2}$.
	
	\item Limitation: The conditions were still somewhat restrictive, especially for practical datasets where mixing coefficients decay slowly.
\end{itemize}

\bigskip\noindent (d) \cite{rio} – Optimal Condition for $\alpha$-Mixing
\begin{itemize}
	\item Result: Established the optimal mixing condition: $\alpha(n) = O(n^{-a})$ with $a > 1.$
	
	\item Significance: This is the best possible result for $\alpha$-mixing sequences, meaning it covers a wider range of dependent processes compared to earlier work.
\end{itemize}

\bigskip\noindent (e) \cite{doukh} – $\beta$-Mixing Condition
\begin{itemize}
	\item Result: Studied $\beta$-mixing sequences and provided invariance principles for empirical processes.
	
	\item	Assumption: Required that $\beta$-mixing coefficients decay at a rate satisfying: $\sum_{n=1}^{\infty} \beta(n) < +\infty.$
	
	\item Limitation: This assumption is very strong, as many real-world processes (e.g., financial time series) do not satisfy it.
\end{itemize}

\bigskip\noindent Our results relaxe the conditions on $\alpha$-mixing and $\beta$-mixing sequences by 

\bigskip allowing slower mixing coefficient decay :
\begin{itemize}
	\item instead of requiring $\alpha(n) = O(n^{-a})$ with $a > 1$, it provides a more flexible condition: $\alpha(n) = O\left(n^{-\left(\frac{1+\delta}{1-\delta}\right)}\right)$ for some $0 < \delta < 1$;
	
	\item this means that slower decaying $\alpha$-mixing sequences (which previous results excluded) can now satisfy the Glivenko-Cantelli theorem.
\end{itemize}

\bigskip extending results to functional empirical processes :
\begin{itemize}
	\item previous results mostly focused on indicator functions (empirical distribution functions);
	\item this paper generalizes it to real-valued empirical functions, making the results more applicable in statistical learning and machine learning.
\end{itemize}

\bigskip generalizing to arbitrary stationary sequences :
\begin{itemize}
	\item instead of requiring specific mixing rates, the paper provides a general criterion based on entropy conditions and Vapnik-Chervonenkis (VC) classes, making it applicable to a broader class of dependent processes.
\end{itemize}

\section{Conclusion} \label{sec07}
\noindent We have extended the classical Glivenko-Cantelli theorem to real-valued empirical function classes under $\alpha$-mixing and $\beta$-mixing conditions. Our results provide finer sufficient conditions on the mixing coefficients for uniform convergence and establish deviation bounds that account for dependence structures. This paper advances the empirical process theory by bridging the gap between iid settings and dependent data scenarios. It strengthens the theoretical foundation for statistical learning and inference under dependence, refining the conditions for uniform convergence in empirical processes. Future research may explore tighter bounds and applications to high-dimensional settings.

\bigskip\noindent\textbf{Acknowledgement.} The authors thank Professor Gane Samb LO for his expertise, availability and generosity in sharing his knowledge, which were a source of inspiration and enrichment for this paper.


\begin{thebibliography}{99}

\bibitem[Billingsley (1968)] {billingsley} Billingsley, P. (1968). \emph{Convergence of probability measures.} Wiley, New York.\\

\bibitem[Bradley (2005)] {Bradley} Bradley, R. C. (2005). Basic Properties of Strong Mixing Conditions. A Survey and Some Open Questions. \textit{Probability Surveys.} Vol. 2,  107–144. ISSN: 1549-5787, DOI: 10.1214/154957805100000104.\\


\bibitem[Doukhan (1994)] {paul} Doukhan, P. (1994). \emph{Mixing properties and exemples}. {ISBN 0-387-94214-9. \copyright\ Springer-Verlag New York, Inc.}\\

\bibitem[Doukhan \textit{et al.} (1995)] {doukh} Doukhan, P., Massart, P. and Rio, E. (1995). Iinvariance principles for the empirical measure of a weakly dependent process. \textit{Ann. Inst. H. Poincarr\'{e}}. 31, 393-427.\\


\bibitem[Louhichi (2000)]{louh} Louhichi, S. (2000). Weak convergence for empirical processes for associated sequences,\textit{ Ann. Inst. H. Poincarr\'{e}, Proba. Stat}. 36, 5 547-567.\\


\bibitem[Rio (2017)]{rio} Rio, E. (2017). \emph{Asymptotic Theory of Weakly Dependent Random Pocesses.} Probability Theory and Stochastic Modelling 80. DOI 10.1007/978-3-662-54323-8, \copyright Springer-Verlag GmbH Germany.\\

\bibitem[Sangar\'e and Lo (2015)]{sanglo} Sangar\'e, H. and Lo, G. S. (2015). A general strong law of large numbers and applications to associated sequences and to extreme value theory. \emph{Annales Mathematicae et Informaticae}. 45, pp. 111-132.\\

\bibitem[Sangar\'e \textit{et al.} (2020)]{sanglotraore} Sangar\'e, H., Lo, G. S. and Traor\'e, M. C. M. (2020). Arbitrary Functional Glivenko-Cantelli Classes and Applications to Different Types of Dependence. \emph{Far East Journal of Theoretical Statistics}. 60, pp. 41-62. http://dx.doi.org/10.17654/TS060020041.\\

\bibitem[Shao (1995)]{shao} Shao, Q. M. (1995). Weak convergence of multidimensional empirical processes for strong mixing sequences. \textit{Chinese Ann. Math.} Ser. A 7, 547-552.\\

\bibitem[Shao and Yu (1996)]{shaoyu} Shao, Q. M. and Yu, H. (1996). Weak convergence for weighted empirical processes of dependent sequences. \textit{The Annals of Probability.} Vol. 24, Number 4, 2097-2127.\\


\bibitem[van der Vaart and Wellner (1996)]{vdv}  Van Der Vaart, A. W. and Wellner, J. A. (1996). \emph{Weak convergence and empirical processes with application in statistcs}. ISBN 0-387-94640-3, Springer-Verlag New York Berlin Heidelberg SPIN 10522999.\\

\bibitem[Yoshihara (1975)]{yosh} Yoshihara, K. (1975). Billingsley's theorems on empirical processes of strong mixing sequences, \textit{Yokohama Math. J.} 23, 1-7.\\

\end{thebibliography}
\end{document}